\documentclass{amsart}
\usepackage{pdfpages}
\newtheorem{theorem}{Theorem}[section]
\newtheorem{lemma}[theorem]{Lemma}

\theoremstyle{definition}
\newtheorem{definition}[theorem]{Definition}
\newtheorem{example}[theorem]{Example}

\theoremstyle{remark}
\newtheorem{remark}[theorem]{Remark}

\numberwithin{equation}{section}

\begin{document}

\title{On some curves with modified orthogonal frame in Euclidean 3-space}

\author{Mohamd Saleem Lone}
\address{International Centre for Theoretical Sciences, Tata Institute of Fundamental Research, 560089, Bengaluru, India}
\curraddr{International Centre for Theoretical Sciences, Tata Institute of Fundamental Research, 560089, Bengaluru, India}
\email{mohamdsaleem.lone@icts.res.in}

\author{Hasan Es}
\address{Gazi University, Gazi Educational Faculty, Department of
Mathematical Education, 06500 Teknikokullar / Ankara-Turkey}
\email{hasanes@gazi.edu.tr}

\author{Murat Kemal Karacan}
\address{Usak University, Faculty of Sciences and Arts, Department of
Mathematics,1 Eylul Campus, 64200,Usak-Turkey}
\email{murat.karacan@usak.edu.tr}

\author{Bahaddin Bukcu}
\address{Gazi Osman Pasa University, Faculty of Sciences and Arts,
Department of Mathematics,60250,Tokat-Turkey}
\email{bbukcu@yahoo.com}






\begin{abstract}
In this paper, we study helices and the Bertrand curves. We obtain some of the classification results of these curves with respect to the modified orthogonal frame in Euclidean 3-spaces.
\end{abstract}

\maketitle

\section{\protect \bigskip Introduction}
In the classical study of differential geometry, curves satisfying particular relations with respect to their curvatures are of greater importance and applications. Helices and Bertrand curves are two among the prominent. A helix is a curve whose tangent makes a constant angle with a fixed direction(axis) \cite{12,h11}. There is a famous classification of a general helix(Lancret theorem): a curve is a general helix iff $\frac{\kappa}{\tau}$= constant \cite{13}. Izumiya and Takeuchi \cite{14} defined a special class of helices called as slant helices, where the normal vector observes a constant angle with the fixed direction. They obtained a necessary and sufficient condition for a curve to be slant helix: a curve is a slant helix iff its geodesic curvature and the  principal normal satisfying the expression
\begin{equation}\label{helix}
\frac{\kappa^2}{(\kappa^2 + \tau^2)^{\frac{3}{2}}}\left(\frac{\tau}{\kappa}\right)^\prime
\end{equation} is a constant function. Kula and Yayl{\i} \cite{14a} studied slant helix and its spherical indicatrix. They showed that a curve of constant precession is a slant helix. Later, Kula et al. \cite{14b} investigated the relations between a general helix and a slant helix.

J. Bertrand in 1850 discovered the notion of Bertrand curve. A Bertrand curve is a curve whose principal normal is normal to some other curve called as Bertrand mate curve. Such a pair of curves is called as Bertrand pair \cite{1,3}. Bertrand curves satisfy a linear relation$(a \tau + b \kappa=1)$ between its curvature and torsion and hence appear as an  analogous form of one dimensional linear Weingarten surfaces which also have found enormous applications  in Computer Aided Geometric Design(CAGD) \cite{11}. Therefore, we see that Bertrand curve appear as a natural generalization of helices \cite{15}, and in particular, Bertrand mates are used as offset curves in CAGD \cite{9}. Schief \cite{10} used the soliton theory to study the geodesic imbedding of Bertrand curves. For more study of helices and Bertrand curves, we refer \cite{h4,5,h3,8}.

\section{Preliminaries}

\bigskip Let $\varphi(s)$ be a $C^{3}$ space curve in Euclidean 3-space $E^3$, parametrized by arc length $s$. We also assume that{ \it its
curvature $\kappa (s)\neq 0$ anywhere}. Then an
orthonormal frame $\left \{ {\bf t,n,b}\right \} $ exists satisfying the
Frenet-Serret equations%
\begin{equation}\label{h2.1}
\left[ 
\begin{array}{c}
{\bf t}^{\prime }(s) \\ 
{\bf n}^{\prime }(s) \\ 
{\bf b }^{\prime }(s)%
\end{array}%
\right] =\left[ 
\begin{array}{ccc}
0 & \kappa & 0 \\ 
-\kappa & 0 & \tau \\ 
0 & -\tau & 0%
\end{array}%
\right] \left[ 
\begin{array}{c}
{\bf t}(s) \\ 
{\bf n}(s) \\ 
{\bf b}(s)%
\end{array}%
\right], 
\end{equation}%
where ${\bf t}$ is the unit tangent, ${\bf n}$ is the unit principal normal, ${\bf b}$ is the unit binormal, and $\tau (s)$ is the torsion. For a given  $C^{1}$ function $\kappa (s)$  and a continuous function $\tau (s)$, there exists a $C^{3}$ curve $\varphi$ which has an orthonormal frame $\left \{{\bf t,n,b}\right \} $ satisfying the Frenet-Serret frame (2.1). Moreover, any other curve $\tilde{\varphi}$ satisfying the same conditions differs from $\varphi$ only by a rigid motion.

Now let $\varphi(t)$ be a general analytic curve which can be reparametrized by its arc length $s$, where $s\in I$ and $I$ is a nonempty open interval.  Assuming that the curvature function has {\it discrete zero points} or $\kappa(s)$ {\it is not identically zero}, we have an orthogonal frame $\left \{ T,N,B\right \} $ defined as follows:%
\begin{equation*}
T=\frac{d\varphi }{ds},\quad  N=\frac{dT}{ds},\quad B=T\times N,
\end{equation*}%
where $T\times N$ is the vector product of $T$ and $N$. The relations
between $\left \{ T,N,B\right \} $ and previous Frenet frame vectors at
non-zero points of $\kappa $ are%
\begin{equation}\label{h2.2}
T={\bf t},N=\kappa {\bf n},B=\kappa {\bf  b}.  
\end{equation}%
Thus, we see that $N(s_{0})=B(s_{0})=0$ when $\kappa (s_{0})=0$ and squares of the length
of $N$ and $B$ vary analytically in $s$. From Eq. (2.2), it is easy to calculate

\begin{equation}\label{h2.3}
\left[ 
\begin{array}{c}
T^{\prime }(s) \\ 
N^{\prime }(s) \\ 
B^{\prime }(s)%
\end{array}%
\right] =\left[ 
\begin{array}{ccc}
0 & 1 & 0 \\ 
-\kappa ^{2} & \frac{\kappa ^{\prime }}{\kappa } & \tau \\ 
0 & -\tau & \frac{\kappa ^{\prime }}{\kappa }%
\end{array}%
\right] \left[ 
\begin{array}{c}
T(s) \\ 
N(s) \\ 
B(s)%
\end{array}%
\right],  
\end{equation}%
where all the differentiation is done with respect to the arc length$(s)$
and 
\begin{equation*}
\tau =\tau (s)=\frac{\det \left( \varphi ^{\prime },\varphi ^{\prime \prime
},\varphi ^{\prime \prime \prime }\right) }{\kappa ^{2}}
\end{equation*}%
is the torsion of $\varphi $. From Frenet-Serret equation, we know that any point, where $\kappa ^{2}=0$ is a removable singularity of $\tau $. Let $\left \langle ,\right \rangle $ be the standard inner product of $E^3$, then $\left \{ T,N,B\right \} $ satisfies:%
\begin{equation}\label{h2.4}
\left \langle T,T\right \rangle =1,\left \langle N,N\right \rangle =\left
\langle B,B\right \rangle =\kappa ^{2},\left \langle T,N\right \rangle
=\left \langle T,B\right \rangle =\left \langle N,B\right \rangle =0.
\end{equation}%
The orthogonal frame defined in (\ref{h2.3}) satisfying (\ref{h2.4}) is called as a modified orthogonal frame \cite{hh7}. We see that for $\kappa=1$, the Frenet-Serret frame coincides with the modified orthogonal frame.

Let $I$ be an open interval of real line $R$ and $M$ be a $n-$dimensional Riemannian manifold and $T_p(M)$ be a tangent space of $M$ at a point $p\in M$. A curve on $M$ is a smooth mapping  $\psi:I\rightarrow M$. As a submanifold of $R$, $I$ has a coordinate system consisting of the identity map $u$ of $I$. The velocity vector of $\psi $ at $s$ $\in I$ is given by
\begin{equation*}
\psi ^{\prime }(s)=\frac{d\psi (u)}{du}\left \vert _{s}\right. \in \text{%
T}_{\psi (s)}(M).
\end{equation*}
A curve $\psi(s)$ is said to be regular if $\psi^\prime(s)\neq 0$ for any $s$. Let $\psi (s)$ be a space curve on $M$ and $\{{\bf t,n,b}\}$ the moving Frenet frame along $\psi $, then we have the following properties 
\begin{equation}\label{h2.5}
\left \{ 
\begin{array}{c}
\psi ^{\prime }(s) =\mathbf{t} \\
D_{\mathbf{t}}\mathbf{t} =\kappa \mathbf{n} \\
D_{\mathbf{t}}\mathbf{n} =-\kappa \mathbf{t}+\tau \mathbf{b} \\
D_{\mathbf{t}}\mathbf{b} =-\tau \mathbf{n},
\end{array}
\right. 
\end{equation}
where $D$ denotes the covariant differentiation on $M$ \cite{h16}.

\section{General helices with modified orthogonal frame}

In this section, we study a curve on manifold $M$. From (\ref{h2.3}), we have 
\begin{equation}\label{h3.1}
\left \{ 
\begin{array}{c}
\psi ^{\prime }(s)=\mathbf{t}=T \\ 
D_TT=N \\ 
D_{T}N=-\kappa ^{2}T+\dfrac{\kappa ^{\prime }}{\kappa }N+\tau B \\ 
D_{T}B=-\tau N+\dfrac{\kappa ^{\prime }}{\kappa }B.
\end{array}%
\right.  
\end{equation}
\begin{theorem}(Lancret theorem)
Let $\psi :I\rightarrow E^{3}$ be a parametrization by arc-length. Then, with respect to the modified orthogonal frame, 
$\psi $ is a general helix if and only if $\dfrac{\tau (s)}{\kappa (s)}=\lambda,$ where $\lambda \in R.$ 
\end{theorem}
\begin{remark} Lancret theorem is a celebrated theorem on helices having many proofs in different ambient spaces and frames. It can be easily proved with the modified orthogonal frame also.
\end{remark}
\begin{theorem}
A unit speed curve $\psi $ is a general helix according to the modified
orthogonal frame if and only if%
\begin{equation}\label{h3.5}
D_{T}D_{T}D_{T}T=\mu N+\frac{3\kappa ^{\prime }}{\kappa }D_{T}N\  \  
\end{equation}%
where%
\begin{equation}\label{h3.6}
\mu =\frac{\kappa ^{\prime \prime }}{\kappa }-\kappa ^{2}-\tau ^{2}\mathtt{-}%
3\left(\frac{\kappa ^{\prime}}{\kappa}\right)^2. 
\end{equation}
\end{theorem}

\begin{proof}
Let $\psi $ be a general helix. From (\ref{h3.1}), we have,%
\begin{equation}\label{h3.7}
D_{T}\left( D_{T}T\right) =-\kappa ^{2}T+\frac{\kappa ^{\prime }}{\kappa }%
N+\tau B  
\end{equation}%
and%
\begin{eqnarray}\label{m12}
\nonumber D_{T}D_{T}D_{T}T &=&-2\kappa \kappa ^{\prime }T-\kappa ^{2}N+\left( \frac{%
\kappa ^{\prime \prime }}{\kappa }-\frac{\kappa ^{\prime ~^{2}}}{\kappa ^{2}}%
\right) N \\
\nonumber &&+\frac{\kappa ^{\prime }}{\kappa }\left( -\kappa ^{2}T+\frac{\kappa
^{\prime }}{\kappa }N+\tau B\right) +\tau ^{\prime }B+\tau \left( -\tau N+%
\frac{\kappa ^{\prime }}{\kappa }B\right).\\
\end{eqnarray}%
Combining like terms of (\ref{m12}), we get
\begin{eqnarray}\label{n1}
 D_{T}D_{T}D_{T}T=\left( -3\kappa \kappa ^{\prime }\right) T+\left( \frac{\kappa ^{\prime \prime }}{\kappa }-\kappa ^{2}-\tau ^{2}\right) N+
\left(2\frac{\kappa^\prime}{\kappa}\tau + \tau^\prime\right)B.
\end{eqnarray}%
Using the third equation of (\ref{h3.1}) in (\ref{n1}), we get 
\begin{eqnarray*}
 D_{T}D_{T}D_{T}T&=&\left( -3\kappa \kappa ^{\prime }\right) T+\left( \frac{\kappa ^{\prime \prime }}{\kappa }-\kappa ^{2}-\tau ^{2}\right) N\\
&&+\frac{1}{\tau}\left(2\frac{\kappa^\prime}{\kappa}\tau + \tau^\prime\right)\left( D_{T}N+\kappa ^{2}T-\frac{\kappa ^{\prime }}{\kappa }N\right).
\end{eqnarray*}
Combining similar terms, we have
\begin{eqnarray}\label{fqb}
\nonumber D_{T}D_{T}D_{T}T&=&\left(-3\kappa \kappa^\prime +\frac{2\tau \kappa^\prime +\tau^\prime \kappa}{\tau}\kappa\right)T\\
\nonumber &&+\left(\frac{\kappa^{\prime\prime}}{\kappa}-\tau^2 -\kappa^2 - \frac{2\tau \kappa^\prime +\tau^\prime\kappa}{\tau\kappa^2}\kappa^\prime\right)N+\left(\frac{2\tau \kappa^\prime
 +\tau^\prime\kappa}{\tau\kappa}\right)D_T N.\\
\end{eqnarray}
Now, since $\psi $ is a general helix, we have
$\frac{\tau }{\kappa }=\textrm{constant}$
and the derivation give rise to%
\begin{equation}\label{n3}
\kappa ^{\prime }\tau =\kappa \tau ^{\prime }\text{ or }\frac{\kappa
^{\prime }}{\kappa }=\frac{\tau ^{\prime }}{\tau }.
\end{equation}%
Substituting (\ref{n3}) in (\ref{fqb}), we get
\begin{equation}\label{h3.9}
D_{T}D_{T}D_{T}T=\left( \frac{\kappa ^{\prime \prime }}{\kappa }-\kappa
^{2}-\tau ^{2}\text{\texttt{-}}\frac{3\kappa ^{\prime ~^{2}}}{\kappa ^{2}}%
\right) N+\frac{3\kappa ^{\prime }}{\kappa }D_{T}N.
\end{equation}
This proves (\ref{h3.5}). Conversely, assume (\ref{h3.5}) is satisfied. We show that the curve $%
\psi $ is a general helix. Differentiating covariantly $N=D_T T$ in (\ref{h3.1}), we get
\begin{equation}\label{qw1}
D_{T}N=D_{T}D_{T}T
\end{equation}%
and so%
\begin{equation}\label{h3.10}
D_{T}D_{T}N=D_{T}D_{T}D_{T}T=\mu N+\frac{3\kappa ^{\prime }}{\kappa }D_{T}N 
\end{equation}
or
\begin{equation*}
D_{T}D_{T}N=\mu N+\frac{3\kappa ^{\prime }}{\kappa }D_{T}N.
\end{equation*}
Using the third equation of (\ref{h3.1}) in above equation, we obtain
\begin{equation*}
D_{T}D_{T}N=\mu N+\frac{%
3\kappa ^{\prime }}{\kappa }\left( -\kappa ^{2}T+\frac{\kappa ^{\prime }}{\kappa 
}N+\tau B\right).
\end{equation*}
Combining the like terms, we get
\begin{equation}\label{h3.11}
D_{T}D_{T}N=-3\kappa\kappa^\prime T + \left(\mu+3\left(\frac{\kappa^\prime}{\kappa}\right)^2\right)N+3\frac{\tau \kappa^\prime}{\kappa}B.  
\end{equation}
Using (\ref{qw1}), the equation in (\ref{n1}) ca be written as
\begin{equation}\label{h3.12}
D_{T}D_{T}N\bigskip =-3\kappa \kappa ^{\prime }T+\left( \frac{\kappa ^{\prime
\prime }}{\kappa }-\kappa ^{2}-\tau ^{2}\right)N +\left( \tau ^{\prime }+\frac{2\tau \kappa ^{\prime }}{\kappa }\right) B.   
\end{equation}%
Consequently from (\ref{h3.11}) and (\ref{h3.12}), we obtain%
\begin{equation*}
\frac{\kappa ^{\prime }}{\kappa }=\frac{\tau ^{\prime }}{\tau }
\end{equation*}%
and so
\begin{equation}\label{h3.13}
\left( \frac{\tau }{\kappa }\right) ^{\prime }=0.  
\end{equation}%
Integrating (\ref{h3.13}) equality, we get%
\begin{equation*}
\frac{\tau }{\kappa }=\textrm{ constant}.
\end{equation*}%
Hence $\psi $ is a general helix. 
\end{proof}

\begin{theorem}
Let $\psi $ be \ a unit speed curve, then $\psi $ is a general helix
according to the modified orthogonal frame if and only if 
\begin{equation}\label{da1}
det(T^{\prime },T^{\prime \prime },T^{\prime \prime \prime })=0.
\end{equation}
\end{theorem}

\begin{proof}
Let $\left( \frac{\tau }{\kappa }\right) $ be constant$.$ We have the
equalities%
\begin{eqnarray*}
\psi ^{^{\prime }} &=&T \\
\psi ^{^{^{\prime \prime }}} &=&T^{\prime }=N \\
\psi ^{^{\prime \prime \prime }} &=&T^{\prime \prime }=-\kappa ^{2}T+\frac{%
\kappa ^{\prime }}{\kappa }N+\tau B
\end{eqnarray*}%
\begin{eqnarray*}
\psi ^{(4)} &=&T^{\prime \prime \prime }=-2\allowbreak \kappa \kappa
^{\prime }T-\kappa ^{2}N\allowbreak +\frac{1}{\kappa }\kappa ^{\prime \prime
}N-\frac{\kappa ^{\prime }\kappa ^{\prime }}{\kappa ^{2}}N\allowbreak \\
&&+\frac{\kappa ^{\prime }}{\kappa }\left( -\kappa ^{2}T+\frac{\kappa
^{\prime }}{\kappa }N+\tau B\right) +\tau \left(- \tau N+\frac{\kappa
^{\prime }}{\kappa }N\right) +\tau ^{\prime }B
\end{eqnarray*}%
\begin{equation*}
\psi^{(4)}=T^{\prime \prime \prime }=-3\allowbreak \kappa \kappa ^{\prime }T+\left( 
\frac{\kappa ^{\prime \prime }}{\kappa }-\kappa ^{2}-\tau ^{2}\right)
N+\left( \frac{2\kappa ^{\prime }\tau }{\kappa }+\tau ^{^{\prime }}\right)
B.\allowbreak
\end{equation*}%
Since $\psi $ is a general helix, i.e, $\frac{\kappa}{\tau}=c$ or $\frac{\kappa^\prime}{\kappa}=\frac{\tau^\prime}{\tau}$. Thus, we have%
\begin{equation*}
\frac{2\kappa ^{\prime }\tau }{\kappa }+\tau ^{\prime }=3\tau ^{\prime }.
\end{equation*}%
Hence, we get%
\begin{equation*}
T^{\prime \prime \prime }=-3\allowbreak \kappa \kappa ^{\prime }T+\left( 
\frac{\kappa ^{\prime \prime }}{\kappa }-\kappa ^{2}-\tau ^{2}\right)
N+3\tau ^{^{\prime }}B\allowbreak .~\  \ 
\end{equation*}%
The above equalities implies that 
\begin{equation}\label{h3.14}
\det (T^{\prime },T^{\prime \prime },T^{\prime \prime \prime })=\left \vert 
\begin{array}{ccc}
0 & 1 & 0 \\ 
-\kappa ^{2} & \frac{\kappa ^{\prime }}{\kappa } & \tau \\ 
-3\mathrm{\allowbreak }\kappa \kappa ^{\prime } & \frac{\kappa ^{\prime
\prime }}{\kappa }-\kappa ^{2}\mathrm{\allowbreak -}\tau ^{2} & 3\tau
^{\prime }%
\end{array}%
\right \vert =\allowbreak 3\kappa \left( \kappa \tau ^{\prime }-\tau \kappa
^{\prime }\right) =3\kappa ^{3}\left( \frac{\tau }{\kappa }\right) ^{\prime
}.  
\end{equation}%
Since $\left( \dfrac{\tau }{\kappa }\right) $ is constant, we have $\left( 
\dfrac{\tau }{\kappa }\right) ^{^{\prime }}=0$ and $\ det(T^{\prime
},T^{\prime \prime },T^{\prime \prime \prime })=0$.

Now, let $det(T^{\prime },T^{\prime \prime },T^{\prime \prime \prime })=0$. It is
clear that $\left( \dfrac{\tau }{\kappa }\right) $ is constant for being%
\begin{equation*}
\left( \frac{\tau }{\kappa }\right) ^{^{\prime }}=0.
\end{equation*}%
Thus, $\psi$ is a general helix.
\end{proof}
\bigskip
\begin{figure}[h]
\centering
  \includegraphics[width=5cm,height=5cm]{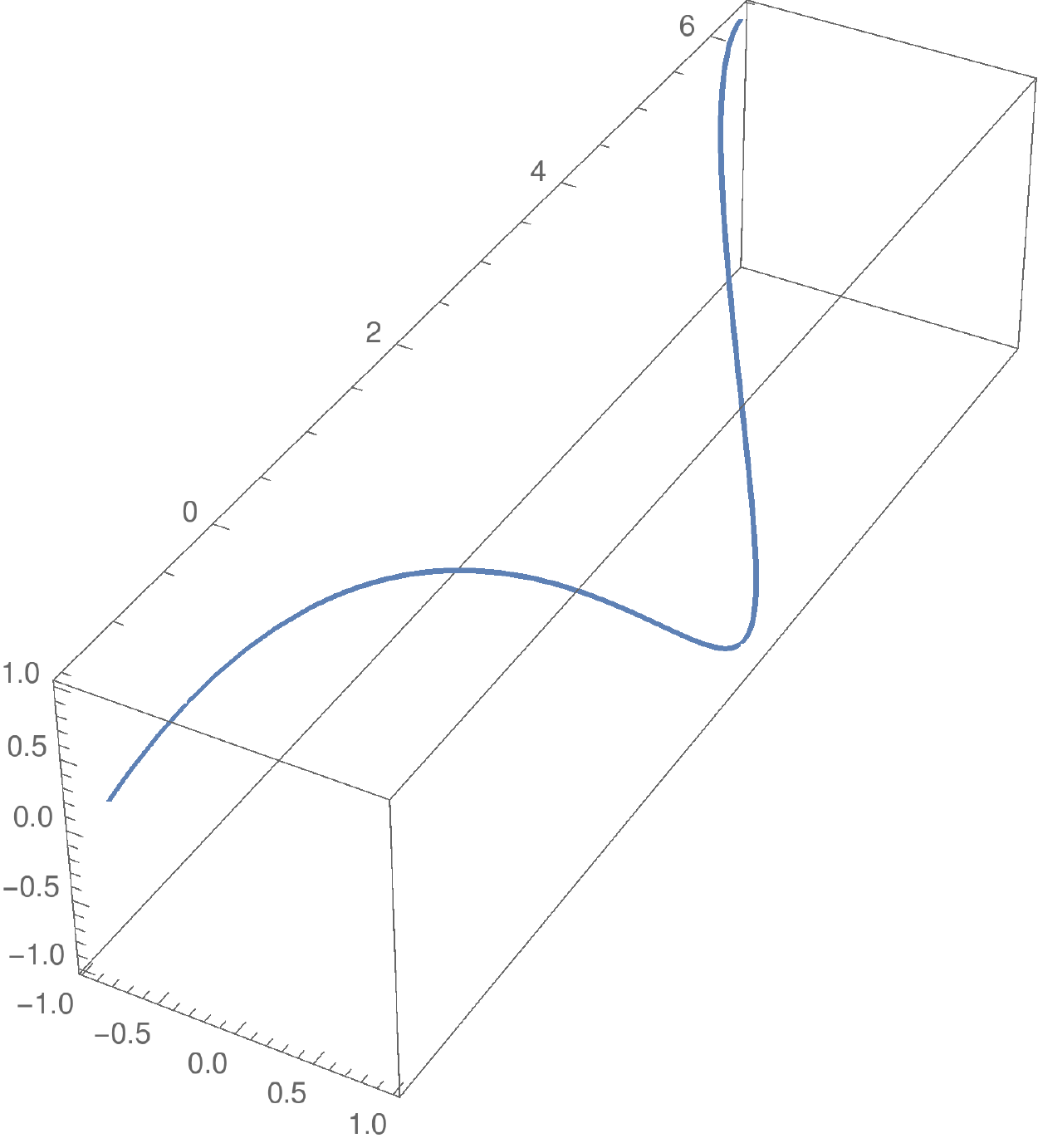}
  \caption{Circular Helix $:\left(\sin \frac{s}{\sqrt{2}}, \frac{s}{\sqrt{2}}, \cos \frac{s}{\sqrt{2}}\right).$}
\end{figure}
\begin{example}
Suppose a curve $\varphi$ be parameterized as
\begin{equation}\label{e1}
\varphi(s)=\left(\sin \frac{s}{\sqrt{2}}, \frac{s}{\sqrt{2}}, 
\cos \frac{s}{\sqrt{2}}\right).
\end{equation}
The Frenet frame vectors $\{{\bf t,n,b}\}$ of (\ref{e1}) are given by
\begin{eqnarray*}
{\bf t}&=&\left( \frac{1}{\sqrt{2}}\cos \frac{s}{\sqrt{2}},\frac{1}{\sqrt{2}},-\frac{1}{\sqrt{2}}\sin \frac{s}{\sqrt{2}}\right).\\
{\bf n} &=& \left( -\sin \frac{s}{\sqrt{2}},0,-\cos \frac{s}{\sqrt{2}}\right).\\
{\bf b} &=&\left( -\frac{1}{\sqrt{2}}\cos \frac{1}{\sqrt{2}},\frac{1}{\sqrt{2}},\frac{1}{\sqrt{2}}\sin \frac{1}{\sqrt{2}}s\right).
\end{eqnarray*}
Also the curvature and torsion is given by
\begin{equation}\label{o}
\kappa(s)=\frac{1}{2}, \quad \tau(s)=\frac{1}{2}.
\end{equation}
Therefore the modified orthogonal frame vectors are given by
\begin{eqnarray*}
T(s)&=&{\bf t}=\left(\frac{1}{\sqrt{2}} \cos \frac{s}{\sqrt{2}}, \frac{1}{\sqrt{2}}, -\frac{1}{\sqrt{2}} \sin \frac{s}{\sqrt{2}}\right).\\
N(s)&=&\kappa {\bf n}=\left(-\frac{1}{2} \sin \frac{s}{\sqrt{2}}, 0, -\frac{1}{2} \cos \frac{s}{\sqrt{2}}\right).\\
B(s)&=&\kappa {\bf b}=\left(-\frac{1}{2\sqrt{2}} \cos \frac{s}{\sqrt{2}}, \frac{1}{2\sqrt{2}}, \frac{1}{2\sqrt{2}} \sin \frac{s}{\sqrt{2}}\right).
\end{eqnarray*}
From the above frame vectors and (\ref{o}), we see that $\varphi$ is arc length parameterized and a helix, respectively. Substituting the above quantities in (\ref{h3.6}), we see that equality in (\ref{h3.5}) is satisfied.

The higher derivative of $T(s)$ are given by
\begin{eqnarray}\label{1t1}
\left \{ 
\begin{array}{c}
T^\prime(s)=\left(-\frac{1}{2} \sin \frac{s}{\sqrt{2}},0,-\frac{1}{2} \cos \frac{s}{\sqrt{2}}\right)\\
T^{\prime \prime}(s)= \left(-\frac{1}{2\sqrt{2}} \cos \frac{s}{\sqrt{2}},0,\frac{1}{2\sqrt{2}} \sin \frac{s}{\sqrt{2}}\right)\\
T^{\prime \prime \prime}(s)=\left(\frac{1}{4} \sin \frac{s}{\sqrt{2}},0,\frac{1}{4} \cos \frac{s}{\sqrt{2}}\right).
\end{array}
\right.
\end{eqnarray}
From (\ref{1t1}), it is easy to verify (\ref{da1}).
\end{example}
\begin{example}
\begin{figure}[h]
\centering
  \includegraphics[width=5cm,height=7cm]{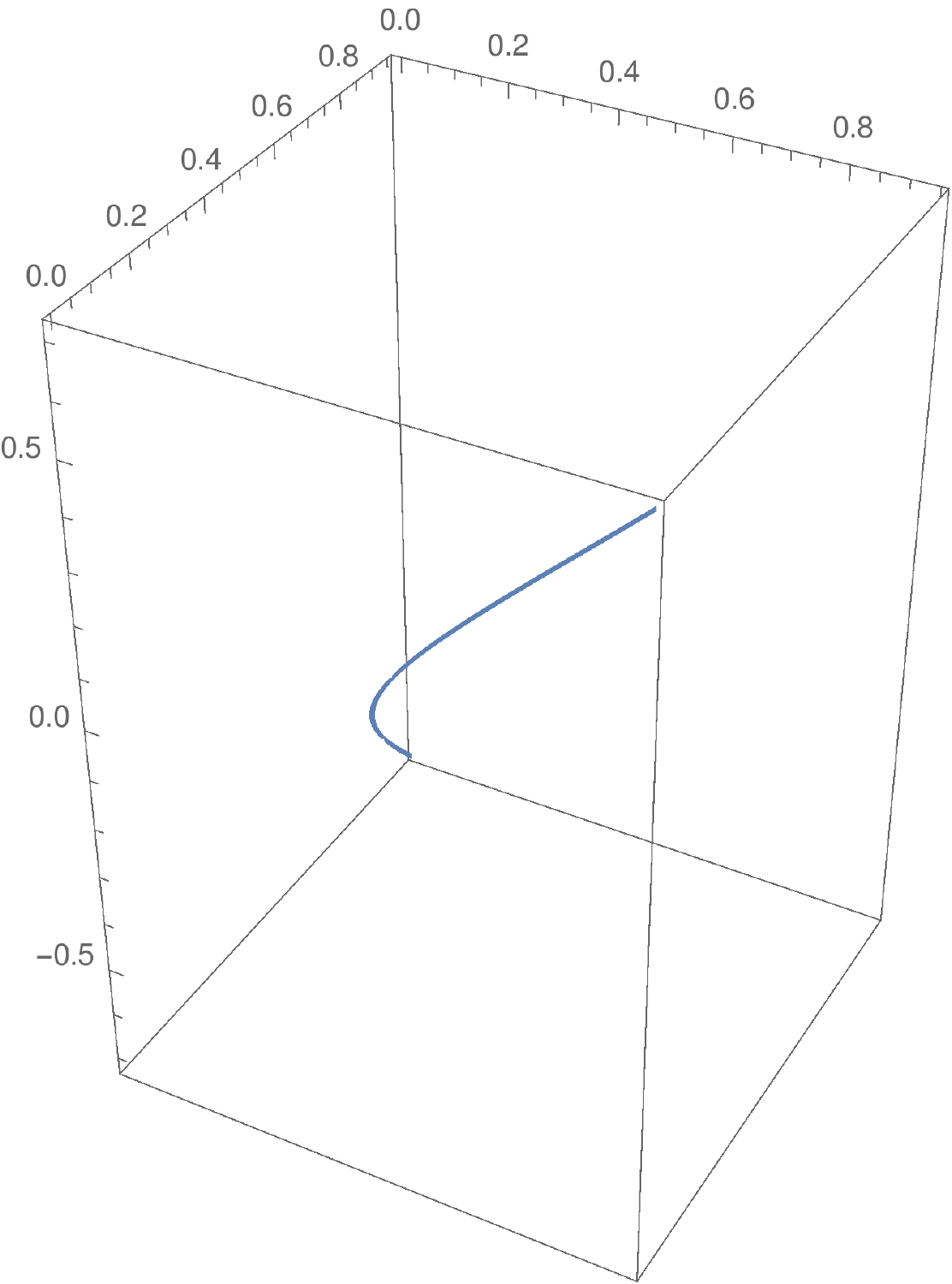}
  \caption{General Helix$:\left( \frac{\left( 1+s\right) ^{\frac{3}{2}}}{3},\frac{\left(
1-s\right) ^{\frac{3}{2}}}{3},\frac{s}{\sqrt{2}}\right).$}
\end{figure}
  Let us consider a curve parameterized as
\begin{equation}\label{lm0}
\psi (s)=\left( \frac{\left( 1+s\right) ^{\frac{3}{2}}}{3},\frac{\left(
1-s\right) ^{\frac{3}{2}}}{3},\frac{s}{\sqrt{2}}\right), 
\end{equation}
where $-1<s<1$.
The Frenet frame vectors $\{{\bf t, n, b}\}$ of (\ref{lm0}) are found as:
\begin{eqnarray*}
\left\{ 
\begin{array}{c}
{\bf t}=\frac{1}{2}\left( \sqrt{1+s},-\sqrt{1-s},\sqrt{2}\right) \\ 
{\bf n}=\frac{1}{\sqrt{2}}\left( \sqrt{1-s},\sqrt{1+s},0\right) \\ 
{\bf b}={\bf t}\times {\bf n}=\frac{1}{2}\left( -\sqrt{1+s},\sqrt{1-s},\frac{\sqrt{2}}{2}
\right) 
\end{array}
\right. 
\end{eqnarray*}
 and $\kappa =\tau =\frac{1}{\sqrt{8\left( 1-s^{2}\right) }}.$
 
Therefore the modified orthogonal frame vectors are given by
\begin{eqnarray*}
T&=&{\bf t}=\left( \frac{\sqrt{1+s}}{2},-\frac{\sqrt{1-s}}{2},\frac{\sqrt{2}}{2}\right).\\
N&=&\kappa {\bf n}= \frac{1}{4}\left( \frac{1}{\sqrt{1+s}},\frac{1}{\sqrt{1-s}},0\right).\\
B&=&\kappa {\bf b}=\frac{1}{4\sqrt{2}}\left( -\frac{1}{\sqrt{1-s}},\frac{1}{\sqrt{1+s}},\frac{\sqrt{2}}{\sqrt{1-s^{2}}}\right). 
\end{eqnarray*}
Also, one can easily find
\begin{eqnarray*}
\dfrac{d\psi }{ds} &=&\mathrm{T}=\frac{1}{2}\allowbreak \left( \sqrt{1+s},-%
\sqrt{1-s},\sqrt{2}\right)  \\
\dfrac{d^{2}\psi }{ds^{2}} &=&T^{\prime }=\frac{1}{4}\left( \frac{1}{\sqrt{%
1+s}},\frac{1}{\sqrt{1-s}},0\right)  \\
\dfrac{d^{3}\psi }{ds^{3}} &=&T^{\prime \prime }=\frac{1}{8}\left( -\frac{1}{%
\left( 1+s\right) ^{\frac{3}{2}}},\frac{1}{\left( 1-s\right) ^{\frac{3}{2}}}%
,0\right)  \\
\dfrac{d^{4}\psi }{ds^{4}} &=&T^{\prime \prime \prime }=\frac{3}{16}\left( 
\frac{1}{\left( 1+s\right) ^{\frac{5}{2}}},\frac{1}{\left( 1-s\right) ^{%
\frac{5}{2}}},0\right) 
\end{eqnarray*}%
Hence $\det (T^{\prime },T^{\prime \prime },T^{\prime \prime \prime })\ $is
zero. This verifies (\ref{da1}).
Again substituting the required quantities in (\ref{h3.6}), equation (\ref{h3.5}) is a straightforward verification. 
\end{example}
\begin{theorem}
Let $\psi :I\rightarrow E^{3}$  be a unit speed curve in $E^3$ such that the curvature  and torsion  of the curve is a non-zero constant and  non-constant, respectively. Then $\psi $ is a slant helix according to modified orthogonal frame if and only if the function 
\begin{equation}\label{h3.15}
\frac{\tau ^{\prime }}{\left( \kappa ^{2}+\tau ^{2}\right) ^{%
\frac{3}{2}}} 
\end{equation}
is constant.
\end{theorem}

\begin{proof}
Let $\psi $ be the given unit speed curve in $E^{3}$. Let $U$ be the vector field
such that the function $\left \langle N(s),U\right \rangle =c$ is constant.
There exist smooth functions $a$ and $b$ such that%
\begin{equation}\label{h3.16}
U=a(s)T(s)+cN(s)+b(s)B(s),\  \ s\in I. 
\end{equation}%
As $U$ is constant, a differentiation of (\ref{h3.16}) together (\ref{h2.3}) gives%
\begin{equation}\label{h3.17}
\left \{ 
\begin{array}{c}
a^{\prime }-c\kappa ^{2}=0 \\ 
a-b\tau \  \ =0 \\ 
b^{\prime }+c\tau =0%
\end{array}%
\right.  
\end{equation}%
From the second equation in (\ref{h3.17}), we get%
\begin{equation}\label{h3.18}
a=\tau b.  
\end{equation}%
Moreover, since $U$ is a constant vector, we have%
\begin{equation}\label{h3.19}
\left \langle U,U\right \rangle =a^{2}+c^{2}\kappa ^{2}+b\kappa ^{2}=\textrm{constant}
\text{.}  
\end{equation}%
We point out that this constraint, together with the second and third equation of
(\ref{h3.17}) is equivalent to the very system (\ref{h3.17}). Combining (\ref{h3.17})
and (\ref{h3.19}), let $m$ be the constant given by%
\begin{equation*}
b^{2}\left( \kappa ^{2}+\tau ^{2}\right) =\text{constant}-c^{2}\kappa
^{2}=m^{2}.
\end{equation*}%
This gives%
\begin{equation*}
b=\pm \frac{m}{\sqrt{\kappa ^{2}+\tau ^{2}}}.
\end{equation*}%
The third equation in (\ref{h3.17}) yields
\begin{equation*}
\frac{d}{ds}\left(\pm\frac{m}{\sqrt{\kappa^2+\tau^2}}\right)=c\tau
\end{equation*}%
on $I$. This can be written as
\begin{equation*}
\pm \frac{c}{m}=\frac{\tau ^{^{\prime }}}{\left( \kappa ^{2}+\tau
^{2}\right) ^{\frac{3}{2}}},\  \  \  \  \  \  \  \  \  \ 
\end{equation*} which proves (\ref{h3.15}).
Conversely, assume that the condition (\ref{h3.15}) is
satisfied. In order to simplify the computations, we assume that the
function in (\ref{h3.15}) is a constant, namely $c$ (the other case is
analogous). Define%
\begin{equation}\label{h3.20}
U=\frac{m\tau }{\sqrt{\kappa ^{2}+\tau ^{2}}}T(s)+\frac{m}{\sqrt{\kappa
^{2}+\tau ^{2}}}B(s)+cN(s),\  \ s\in I. 
\end{equation}%
A differentiation of (\ref{h3.20}) together with the modified orthogonal frame gives 
\begin{equation*}
U^{^{\prime }}=\dfrac{m\tau ^{\prime }\kappa ^{2}}{\left( \kappa ^{2}+\tau
^{2}\right) ^{\frac{3}{2}}}T(s)-\dfrac{m\tau \tau ^{\prime }}{\left( \kappa
^{2}+\tau ^{2}\right) ^{\frac{3}{2}}}B(s)+\dfrac{m\tau }{\sqrt{\kappa
^{2}+\tau ^{2}}}N(s)
\end{equation*}%
\begin{equation*}
+\dfrac{m}{\sqrt{\kappa ^{2}+\tau ^{2}}}\left( -\tau N(s)\right) +\dfrac{%
m\tau ^{\prime }}{\left( \kappa ^{2}+\tau ^{2}\right) ^{3/2}}\left(-\kappa
^{2}T(s)+\underbrace{\frac{\kappa^\prime}{\kappa}N}_0+\tau B(s)\right)=0
\end{equation*}%
that is, $U$ is a constant vector. On the other hand,%
\begin{equation*}
\left \langle N(s),U\right \rangle =\left \langle N(s),\frac{m\tau }{\sqrt{\kappa
^{2}+\tau ^{2}}}T(s)+cN(s)+\frac{m}{\sqrt{\kappa ^{2}+\tau ^{2}}}%
B(s)\right \rangle =c.
\end{equation*}%
This means that $\psi $ is a slant helix. 
\begin{remark}
For $\kappa=1$, we see that with respect to the modified orthogonal frame, the condition for slant helix in (\ref{h3.15}) is equal to the  condition for slant helix with respect to Frenet-Serret frame in (\ref{helix}). 
\end{remark}
\begin{example} 
Salkowski \cite{8j} introduced a family of curves with 
constant curvature but non-constant torsion and called them as  Salkowski curves. Later, Monterde \cite{5b} characterized Salkowski curves as the space curves with constant curvature and whose normal vector makes a constant angle with a fixed line. Thus, we can give Salkowski curves as striking example of slant helices.
\end{example}
\end{proof}
\section{Bertrand curves with modified orthogonal frame}
\begin{definition}\label{def1}Let $\varphi$ and $\psi$ be two curves with non-vanishing curvatures and torsions and $N_{\varphi}$ and $N_{\psi}$ are normals of $\varphi$ and $\psi$, respectively. If $N_{\varphi}$ and $N_{\psi}$ are parallel, then 
 $\left( \varphi \ ,\psi \right) $ is called Bertrand pair \cite{3}. 
\end{definition}
From the definition of Bertrand pair $\left( \varphi(s) \ ,\psi(s^\ast) \right) $, there is a functional relation $s^\ast =s^\ast(s)$
such that $$\delta^\ast(s^\ast(s))=\delta(s).$$
Let
$\left( \varphi \ ,\psi \right) $ be a Bertrand pair, we can write 
\begin{equation}\label{m1}
\psi (s)=\varphi (s)+\delta (s)N_{\varphi}(s).\end{equation}
\begin{theorem}\label{thm1}The distance between the corresponding points of a Bertrand pair $\left( C_{\varphi}  ,C_{\psi} \right) $ with respect to the modified orthogonal frame is constant.
\end{theorem}
\begin{proof}
Let $C_{\varphi }$ and $C_{\psi}$ be given $\varphi(s)$ and $\psi(s)$, respectively. Using (\ref{m1}), we can write
\begin{equation*}
\psi=\varphi+\delta N_{\varphi }.
\end{equation*}%
Differentiating, and using modified orthogonal frame, we obtain%
\begin{equation}\label{3.1}
\psi ^{\prime }(s)=\left( 1-\delta \kappa _{\varphi }^{2}\right) T_{\varphi
}+\left( \delta ^{\prime }+\frac{\delta \kappa _{\varphi }^{\prime }}{%
\kappa _{\varphi }}\right) N_{\varphi }+\delta \tau _{\varphi }B_{\varphi }. 
\end{equation}%
Taking the inner product (\ref{3.1}) with $N_{\varphi}$, and using $N_{\varphi}\parallel N_{\psi}$, we obtain
\begin{equation*}
\left \langle N_{\varphi },\psi ^{\prime }(s)\right \rangle =\delta
^{\prime }+\frac{\delta \kappa _{\varphi }^{\prime }}{\kappa _{\varphi }}=0.
\end{equation*}%
This implies
\begin{equation}\label{3.2}
-\frac{\delta ^{\prime }}{\delta }=\frac{\kappa _{\varphi }^{\prime }}{%
\kappa _{\varphi }}.  
\end{equation}%
Integrating (\ref{3.2}), we get 
\begin{equation}\label{3.3}
\delta (s)=\frac{c}{\kappa _{\varphi }(s)},  
\end{equation}%
where $c$ a real number except zero,
Therefore%
\begin{equation*}
\left \Vert \psi -\varphi (s)\right \Vert =\left \Vert\frac{c}{\kappa _{\varphi}(s)}N_{\varphi }(s)\right \Vert = \left \vert c\right \vert\left \vert \frac{1}{\kappa_\varphi}\right \vert\left \vert \kappa_\varphi\right \vert        = \left \vert c\right \vert .
\end{equation*}
\end{proof}
\begin{lemma}\label{lem1}
The angle between tangent lines of a Bertran pair $(\varphi, \psi)$ is constant.
\end{lemma}
\begin{proof}
Differentiaiting $\left \langle T_{\psi },T_{\varphi }\right \rangle$, we get  
\begin{equation*}
\frac{d}{ds}\left \langle T_{\psi },T_{\varphi }\right \rangle =\left
\langle T_{\psi },N_{\varphi }\right \rangle +\left \langle T_{\psi
},N_{\varphi }\right \rangle .
\end{equation*}%
Since $\frac{N_{\psi }}{\kappa _{\psi }}=\pm \frac{N_{\varphi }}{\kappa
_{\varphi }},$ $\left \langle N_{\psi },T_{\varphi }\right \rangle =0$ and $%
\left \langle N_{\varphi },T_{\psi }\right \rangle =0$ $.$ Thus we get $%
\frac{d}{ds}\left \langle T_{\varphi },T_{\psi }\right \rangle =0$ and so $%
\left \langle T_{\varphi },T_{\psi }\right \rangle =$constant.
\end{proof}
\begin{theorem}A pair of curves $(C_{\varphi },C_{\psi})$ with $\tau _{\varphi }\neq 0$ is a Bertrand pair if and only if
\begin{equation}\label{3.4}
c\kappa _{\varphi }+ac\tau _{\varphi }=1,\text{ } 
\end{equation}%
where $a=\cot \theta $ and $c=$constant.
\end{theorem}
\begin{proof}
Suppose $(C_{\varphi }:\varphi,C_{\psi}:\psi)$ be a Bertrand pair with $\tau_{\varphi}\neq0$. Let the modified orthogonal frames of $\varphi(s)$ and $\psi(s)$ are given by
\begin{equation*}
\left \{ T_{\varphi }={\bf t}_{\varphi },N_{\varphi }=\kappa _{\varphi }{\bf n}_{\varphi
},B_{\varphi }=\kappa _{\varphi }{\bf b}_{\varphi }\right \} ,\left \{ T_{\psi
}={\bf t}_{\psi },N_{\psi }=\kappa _{\psi }{\bf n}_{\psi },B_{\psi }=\kappa _{\psi
}{\bf b}_{\psi }\right \},
\end{equation*}%
respectively.
From Lemma \ref{lem1}, we know that the angle $\theta $ between $T_{\varphi }$
and $T_{\psi }$ is constant. Thus, we may write 
\begin{equation}\label{3.5}
T_{\psi }=\cos \theta T_{\varphi }+\frac{\sin \theta }{\kappa _{\varphi }}%
B_{\varphi }.  
\end{equation}%
Using (\ref{3.2}) and (\ref{3.3}) in (\ref{3.1}), we get 
\begin{equation}\label{3.6}
\psi^\prime(s) =\left( 1-c\kappa _{\varphi }\right) T_{\varphi
}+\frac{c\tau _{\varphi }}{\kappa _{\varphi }}B_{\varphi }. 
\end{equation}%
From (\ref{3.5}) and (\ref{3.6}), we obtain%
\begin{equation*}
\frac{1-c\kappa _{\varphi }}{\cos \theta }=\frac{\frac{c\tau _{\varphi }}{%
\kappa _{\varphi }}}{\frac{\sin \theta }{\kappa _{\varphi }}}
\end{equation*}%
or 
\begin{equation}\label{3.7}
c\kappa _{\varphi }+ac\tau _{\varphi }=1. 
\end{equation}
Conversly, suppose $C_{\varphi}$ be a curve satisfying $c\kappa _{\varphi }+ac\tau
_{\varphi }=1$ and $\tau _{\varphi }\neq 0$. Define another curve $C_{\psi}$ as:
\begin{equation*}
\psi (s)=\varphi (s)+\delta (s)N_{\varphi }(s).
\end{equation*}%
We shall prove that $C_{\varphi }$ and $C_{\psi }$ are Bertrand mates. From (\ref{3.6}), we know
\begin{equation}\label{3.8}
\psi ^{\prime }(s)=\left( 1-c\kappa _{\varphi }\right) T_{\varphi }+\frac{%
c\tau _{\varphi }}{\kappa _{\varphi }}B_{\varphi }.  
\end{equation}%
Using the given condition in (\ref{3.8}), we obtain
\begin{equation*}
\psi ^{\prime }(s)=ac\tau _{\varphi }T_{\varphi }+\frac{c\tau _{\varphi }}{%
\kappa _{\varphi }}B_{\varphi }=c\left( aT_{\varphi }+\frac{B_{\varphi }}{\kappa
_{\varphi }}\right) \tau _{\varphi }\text{.}
\end{equation*}%
 Hence the tangent vector to $%
C_{\psi }$ is given by%
\begin{equation}\label{3.9}
T_{\psi }=\frac{\psi ^{\prime }(s)}{\left \Vert \psi ^{\prime }(s)\right
\Vert }=\frac{c\tau _{\varphi }}{\left \vert c\tau _{\varphi }\right \vert }%
\frac{\left( aT_{\varphi }+\frac{B_{\varphi }}{\kappa _{\varphi }}\right) }{%
\sqrt{a^{2}+1}}=\pm \frac{aT_{\varphi }+\frac{B_{\varphi }}{\kappa _{\varphi }}}{%
\sqrt{a^{2}+1}};a=const. 
\end{equation}%
Differentiation (\ref{3.9}) with respect to $s^\ast$, we get%
\begin{equation}
\frac{dT_{\psi }}{ds^{\ast }}=\pm \frac{1}{\sqrt{a^{2}+1}}\left \{
aN_{\varphi }-\frac{\kappa _{\varphi }^{\prime }}{\kappa _{\varphi }^{2}}%
B_{\varphi }+\frac{1}{\kappa _{\varphi }}\left( -\tau _{\varphi }N_{\varphi }+%
\frac{\kappa _{\varphi }^{\prime }}{\kappa _{\varphi }}B_{\varphi }\right)
\right \} \frac{ds}{ds^{\ast }}.  \notag
\end{equation}%
This implies that
\begin{equation*}
N_{\psi }=\pm \frac{a-\frac{\tau _{\varphi }}{\kappa _{\varphi }}}{\sqrt{%
a^{2}+1}}\frac{ds}{ds^{\ast }}N_{\varphi }.  
\end{equation*}%
or \ 
\begin{equation}\label{3.10}
\frac{N_{\psi }}{\kappa _{\psi }}=\pm \frac{a-\frac{\tau _{\varphi }}{%
\kappa _{\varphi }}}{\sqrt{a^{2}+1}}\frac{ds}{ds^{\ast }}\frac{\kappa
_{\varphi }}{\kappa _{\psi }}\frac{N_{\varphi }}{\kappa _{\varphi }}.
\end{equation}%
Since $\frac{N_{\psi }}{\kappa _{\psi }} \textrm{ and }\pm \frac{N_{\varphi }}{\kappa
_{\varphi }}$ are unit vectors, from (\ref{3.10}), we have%
\begin{equation*}
\frac{ds^{\ast }}{ds}=\frac{\kappa _{\varphi }}{\kappa _{\psi }}\frac{a-%
\frac{\tau _{\varphi }}{\kappa _{\varphi }}}{\sqrt{a^{2}+1}}
\end{equation*}%
and \ 
\begin{equation*}
\frac{N_{\psi }}{\kappa _{\psi }}=\pm \frac{N_{\varphi }}{\kappa _{\varphi }}%
.
\end{equation*}
This completes the proof.
\end{proof}
\begin{theorem}
Let $(C_{\varphi },C_{\psi })\ $be a Bertrand mate in Euclidean 3-space $%
E^{3}$ according to the modified orthogonal frame, then the following identities hold:
\end{theorem}
\begin{eqnarray}\label{id1}\left\{
\begin{array}{c}
\kappa _{\varphi } =\frac{c\kappa _{\psi }+\sin ^{2}\theta }{c\left(
1+c\kappa _{\psi }\right) } \vspace{.2cm}\\
\tau _{\varphi }\tau _{\psi } =\left( \frac{\sin \theta }{c}\right) ^{2}>0.
\end{array}
\right.
\end{eqnarray}
\begin{proof}
From (\ref{3.5}) and (\ref{3.6}), we can write respectively, 
\begin{equation}\label{3.11}
\left \{ 
\begin{array}{c}
\dfrac{1-c\kappa _{\varphi }}{\cos \theta }=\dfrac{ds^{\ast }}{ds}\Rightarrow 
\dfrac{ds^{\ast }}{ds}\cos \theta =1-c\kappa _{\varphi } \\ 
\dfrac{c\dfrac{\tau _{\varphi }}{\kappa _{\varphi }}}{\frac{\sin \theta }{%
\kappa _{\varphi }}}=\dfrac{ds^{\ast }}{ds}\Rightarrow \dfrac{ds^{\ast }}{ds}%
\sin \theta =c\tau _{\varphi }.%
\end{array}%
\right. 
\end{equation}%
Interchanging the roles of $\varphi $ and $\psi $, thus $(\psi ,\varphi )$ is also a Bertrand mate and in this case, $\delta $ and $\theta $ are replaced
with $-\delta $ and $-\theta ,$ respectively$.$ Hence we can write as%
\begin{equation}\label{3.12}
\left \{ 
\begin{array}{c}
\dfrac{ds}{ds^{\ast }}\cos \theta =1+c\kappa _{\psi } \\ 
\dfrac{ds}{ds^{\ast }}\sin \theta =c\tau _{\psi }.%
\end{array}%
\right.  
\end{equation}%
By multiplying the first parts of (\ref{3.11}) and (\ref{3.12}), we
get
\begin{equation*}
\cos ^{2}\theta =\left( 1-c\kappa _{\varphi }\right) \left( 1+c\kappa _{\psi
}\right)
\end{equation*}%
or 
\begin{equation}
\kappa _{\varphi }=\frac{c\kappa _{\psi }+\sin ^{2}\theta }{c\left(
1+c\kappa _{\psi }\right) }.\text{ } 
\end{equation}%
Now multiplying the second parts of (\ref{3.11}) and (\ref{3.12}), we get%
\begin{equation}
\sin ^{2}\theta =c^{2}\left( \tau _{\varphi }\tau _{\psi }\right) . 
\end{equation}%
From above equation, we see that 
\begin{equation*}
\tau _{\varphi }\tau _{\psi }>0.
\end{equation*}
\end{proof}
\begin{example}
Let $\varphi$ be a curve parameterized by
\begin{equation}\label{x1}
\varphi(s)=\left(\frac{1}{2\sqrt{2}}\sin 2s, \frac{1}{2\sqrt{2}}\cos 2s,\frac{s}{\sqrt{2}}\right).
\end{equation}
The modified orthogonal frame vectors for (\ref{x1}) are given by
\begin{eqnarray}
\label{x7}T_\varphi &=&\left(\frac{1}{\sqrt{2}}\cos 2s, -\frac{1}{\sqrt{2}}\sin 2s, \frac{1}{\sqrt{2}}\right).\\
\label{x8}N_\varphi &=&\left(-\sqrt{2} \sin 2s, -\sqrt{2}\cos 2s,0\right).\\
\label{x9}B_\varphi &=&\left(-2 \sin 2s,-2 \cos 2s, 0\right).
\end{eqnarray}
From above, we see that $\varphi$ is parameterized by arc length. Also, the curvature and torsion is given by:
$$\kappa_\varphi =\sqrt{2},\quad \tau_\varphi =-\sqrt{2}.$$
Now for $c=\frac{1}{\sqrt{2}}$, with the help of (\ref{m1}), we can construct a  curve  $\psi$ parameterized as
\begin{equation}\label{x3}
\psi(s)=\left(-\frac{1}{2\sqrt{2}}\sin 2s, -\frac{1}{2\sqrt{2}}\cos 2s, \frac{s}{\sqrt{2}}\right).
\end{equation}
The modified orthogonal frame vectors of (\ref{x3}) are found as
\begin{eqnarray}\label{x4}
T_\psi &=&\left(-\frac{1}{\sqrt{2}}\cos 2s, \frac{1}{\sqrt{2}}\sin 2s,\frac{1}{\sqrt{2}}\right).\\
\label{x5} N_\psi &=&\left(\sqrt{2}\sin 2s,\sqrt{2}\cos 2s,0\right).\\
\label{x6} B_\psi &=&\left(-\cos 2s,\sin 2s,-1\right).
\end{eqnarray}
From (\ref{x8}) and (\ref{x5}), we see that $(\varphi, \psi)$ is a Bertrand pair. The curvature and the torsion of $\psi(s)$ are given by $$\kappa_\psi =\sqrt{2},\quad \tau_\psi =-\sqrt{2}.$$ From (\ref{x7}) and (\ref{x4}), we have $\theta=\frac{\pi}{2}$. Hence (\ref{3.4}) is straightforward. Again substituting the required quantities, the identities in (\ref{id1}) are direct verifications.
\begin{figure}[h]
\centering
  \includegraphics[width=3cm,height=5cm]{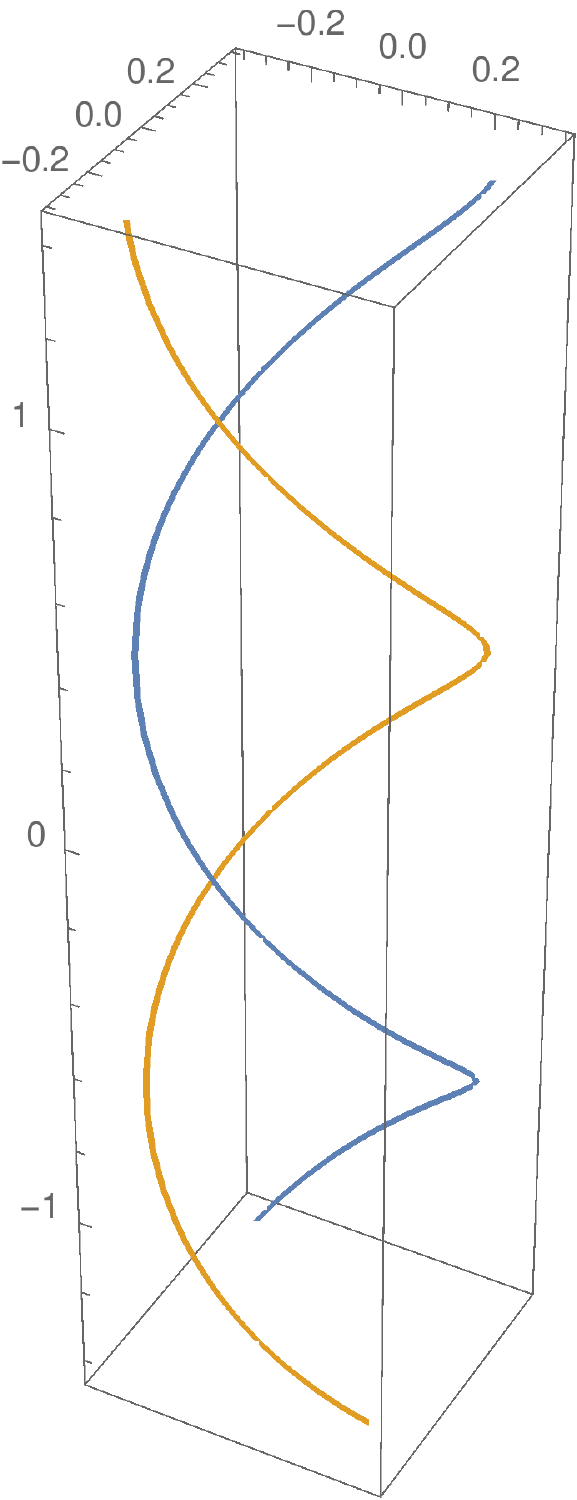}
  \caption{Bertrand pair $(\varphi,\psi)$}
\end{figure}
\end{example}
\begin{theorem}\label{thm4}
Suppose there exists a one-one relation between the points of
the curves $C_{\varphi }$ and $C_{\psi }$, such
that at the corresponding points $P_{\varphi }$ on $C_{\varphi }$ and $P_{\psi }$
on $C_{\psi }:$\\
\textbf{(a)} $\kappa _{\varphi }$ is constant.\\
\textbf{(b)} $\tau _{\psi }$  is constant.\\
\textbf{(c)} $T_{\varphi }$ is
parallel to $T_{\psi }$, \newline
then
the curve $C$ generated by $P$ that divides the segment $P_{\varphi }P_{\psi
}$ in ratio $h:1$ is a Bertrand curve.
\end{theorem}
\begin{proof}
Let  $\alpha(s),$ $\alpha _{\varphi }(s),$ $\alpha _{\psi }(s)$ be the coordinate vectors at
the points $P$, $P_{\varphi }$, $P_{\psi }$ on the curves 
$C$, $C_{\varphi }$, $C_{\psi }$ respectively. Then from a convex combination of points $P_{\varphi
} $ and $P_{\psi }$ , the equation of point $P$ is 
\begin{equation}\label{3.15}
\alpha (s)=h\alpha _{\varphi }(s)+(1-h)\alpha _{\psi }(s),\, \,h\, \in R,h\,
\in \left[ 0,1\right] .  \
\end{equation}%
Differentiating (\ref{3.15}) with respect to $s$ and using the hypothesis, i.e., $T_{\varphi
}=T_{\psi }\,,\,$we find 
\begin{equation}\label{3.16}
\, \, \,Tds=hT_{\varphi }ds_{\varphi }+(1-h)T_{\psi }ds_{\psi }=\left[
hds_{\varphi }+(1-h)ds_{\psi }\right] T_{\varphi } 
\end{equation}%
\begin{equation}\label{3.17}
T=T_{\varphi }=T_{\psi }  
\end{equation}%
\begin{equation}\label{3.18}
h\frac{ds_{\varphi }}{ds}+(1-h)\frac{ds_{\psi }}{ds}=1.
\end{equation}%
Similary, by differentiating (\ref{3.17}), we obtain%
\begin{equation}\label{3.19}
Nds=N_{\varphi }ds_{\varphi }=N_{\psi }ds_{\psi },\text{ }  
\end{equation}%
\begin{equation}\label{3.20}
\kappa ds=\kappa _{\varphi }ds_{\varphi }=\kappa _{\psi }ds_{\psi }, 
\end{equation}%
and%
\begin{equation}\label{3.21}
\frac{ds_{\varphi }}{ds}=\frac{\kappa }{\kappa _{\varphi }}. 
\end{equation}%
\begin{equation}\label{3.22}
\frac{N}{\kappa }=\frac{N_{\varphi }}{\kappa _{\varphi }}=\frac{N_{\psi }}{%
\kappa _{\psi }}.  
\end{equation}%
From the vector product of (\ref{3.17}) and (\ref{3.22}), we have%
\begin{equation}\label{3.23}
\frac{B}{\kappa }=\frac{B_{\varphi }}{\kappa _{\varphi }}=\frac{B_{\psi }}{%
\kappa _{\psi }}.  
\end{equation}%
$\allowbreak $Differentiating (\ref{3.23}) and using
(\ref{3.17}),(\ref{3.22}), we get
\begin{equation}\label{re1}
\tau \frac{N}{\kappa }=\tau _{\psi }\frac{N_{\psi }}{\kappa _{\psi }}%
\frac{ds_{\psi }}{ds}
\end{equation}%
and 
\begin{equation}\label{3.24}
\frac{ds_{\psi }}{ds}=\frac{\tau }{\tau _{\psi }}. 
\end{equation}%
By inserting (\ref{3.21}) and (\ref{3.24}) in (\ref{3.18}), we get 
\begin{equation*}
\left( \frac{h}{\kappa _{\varphi }}\right) \kappa +\left( \frac{1-h}{\tau
_{\psi }}\right) \tau =1;\kappa _{\varphi }\neq 0,\tau _{\psi }\neq 0,
\end{equation*}%
which is the desired result, since $h,\kappa _{\varphi },\tau _{\psi }$ are
constant. 
\end{proof}
\begin{theorem}
If the condition \textbf{(c)} in Theorem \ref{thm4} is modified so that at the corresponding
points $P_{\varphi }$ and $\ P_{\psi }$ the binormals $B_{\varphi }$ and $%
B_{\psi }$ are parallel, then the curve $C$ is a Bertrand curve.
\end{theorem}
\begin{proof}
Since $\dfrac{B_{\varphi }}{\kappa _{\varphi }}=\dfrac{B_{\psi }}{\kappa
_{\psi }}$, similar as in (\ref{re1}) and (\ref{3.24}), we obtain 
\begin{equation}\label{3.25}
\frac{N_{\varphi }}{\kappa _{\varphi }}=\frac{\tau _{\psi }}{\tau _{\varphi }}%
\frac{ds_{\psi }}{ds_{\varphi }}\frac{N_{\psi }}{\kappa _{\psi }}. 
\end{equation}
and 
\begin{equation}\label{3.26}
\frac{ds_{\psi }}{ds_{\varphi }}=\frac{\tau _{\varphi }}{\tau _{\psi }}. 
\end{equation}
Using (\ref{3.26}), we can rewrite (\ref{3.25}) as 
\begin{equation*}
\frac{N_{\varphi }}{\kappa _{\varphi }}= \frac{N_{\psi }}{\kappa _{\psi }}%
.
\end{equation*}%
Thus, we have 
\begin{equation*}
T_{\varphi }=\dfrac{N_{\varphi }}{\kappa _{\varphi }}\times \dfrac{B_{\varphi }}{%
\kappa _{\varphi }}=\dfrac{1}{\kappa _{\varphi }^{2}}\left( \dfrac{\kappa
_{\varphi }}{\kappa _{\psi }}N_{\psi }\times \dfrac{\kappa _{\varphi }}{%
\kappa _{\psi }}B_{\psi }\right) =\frac{N_{\psi }}{\kappa _{\psi }}%
\times \frac{B_{\psi }}{\kappa _{\psi }}=T_{\psi }.
\end{equation*}%
Therefore $T_{\varphi }$ is parallel to $T_{\psi }$, so that by Theorem \ref{thm4}, $C$
is a Bertrand curve .
\end{proof}
\begin{theorem}
If the condition \textbf{(c)} in Theorem \ref{thm4} is modified so that at the corresponding
points $P_{\varphi }$ and $\ P_{\psi }$ the tangent $T_{\varphi }$ at $%
P_{\varphi }$ is parallel to the binormal $B_{\psi }$ at $P_{\psi }$, then
the curve $C$ is a Bertrand curve.
\end{theorem}
\begin{proof}
Since $T_{\varphi }=\dfrac{B_{\psi }}{\kappa _{\psi }}$, it follow that
$\allowbreak $%
\begin{equation}\label{3.27}
N_{\varphi }=-\frac{\tau _{\psi }}{\kappa _{\psi }}\frac{ds_{\psi }}{%
ds_{\varphi }}N_{\psi }.\  \  \  \  \  \  \  \  \  \  \  \  \  \  \  \  \   
\end{equation}
\ Hence $N_{\varphi }$ is parallel to $N_{\psi }$ and since $\frac{N_{\varphi
}}{\kappa _{\varphi }}$ and $\frac{N_{\psi }}{\kappa _{\psi }}$ are unit
vectors, 
\begin{equation}\label{3.28}
\frac{N_{\varphi }}{\kappa _{\varphi }}=\pm \frac{N_{\psi }}{\kappa _{\psi }}
\end{equation}%
From (\ref{3.27}) and (\ref{3.28}), we get
\begin{equation}\label{3.29}
\frac{ds_{\psi }}{ds_{\varphi }}=-\frac{\kappa _{\varphi }}{\tau _{\psi }}\ .
\end{equation}
Moreover since%
\begin{equation*}
B_{\varphi }=T_{\varphi }\times N_{\varphi }=\dfrac{B_{\psi }}{\kappa _{\psi }%
}\times \frac{\kappa _{\varphi }}{\kappa _{\psi }}N_{\psi }=\frac{\kappa
_{\varphi }}{\kappa _{\psi }^{2}}\left( B_{\psi }\times N_{\psi }\right) =%
\frac{\kappa _{\varphi }\kappa _{\psi }^{2}}{\kappa _{\psi }^{2}}\left(
-T_{\psi }\right) =-\kappa _{\varphi }T_{\psi }\text{,}
\end{equation*}%
we get 
\begin{equation*}
T_{\psi }=-\frac{B_{\varphi }}{\kappa _{\varphi }}.
\end{equation*}%
Let \ $R,$ $R_{\varphi },$ $R_{\psi }$ be the coordinate vectors at the
points $P$, $P_{\varphi }$, $P_{\psi }$ on the curves $C$, $C_{\varphi }$, $%
C_{\psi }$ respectively. Then%
\begin{equation}\label{3.31}
R=hR_{\varphi }+(1-h)R_{\psi }.
\end{equation}%
By differentiating (\ref{3.31}) with respect to $s$ and by (\ref{3.29}), we have 
\begin{equation}\label{3.32}
T=\left[ hT_{\varphi }-\frac{\kappa _{\varphi }}{\tau _{\psi }}(1-h)T_{\psi }%
\right] \frac{ds_{\varphi }}{ds}.  
\end{equation}%
Taking the norm of (\ref{3.32}), we obtain 
\begin{equation}\label{3.33}
\frac{ds_{\varphi }}{ds}=\frac{\tau _{\psi }}{\sqrt{h^{2}\tau _{\psi
}^{2}+\kappa _{\varphi }^{2}(1-h)^{2}}}. 
\end{equation}%
Thus with the help of (\ref{3.33}), we can rewrite (\ref{3.32}) as 
\begin{equation*}
T=\left[ \frac{h\tau _{\psi }}{\sqrt{h^{2}\tau _{\psi }^{2}+\kappa
_{\varphi }^{2}(1-h)^{2}}}\right] T_{\varphi }+\left[ \frac{-(1-h)\kappa
_{\varphi }}{\sqrt{h^{2}\tau _{\psi }^{2}+\kappa _{\varphi }^{2}(1-h)^{2}}}%
\right] T_{\psi }
\end{equation*}%
or%
\begin{equation}\label{3.34}
\ T=h_{\varphi }T_{\varphi }+h_{\psi }T_{\psi }\ ,\  \   
\end{equation}%
where
$\allowbreak $ 
\begin{eqnarray*}
h\frac{ds_{\varphi }}{ds} &=&\frac{h\, \, \tau _{\psi }}{\sqrt{h^{2}\tau
_{\psi }^{2}+\kappa _{\varphi }^{2}(1-h)^{2}}}=h_{\varphi },\, \,h_{\varphi }=%
\text{constant} \\
(1-h)\frac{ds_{\psi }}{ds} &=&-\, \frac{\, \left( 1-h\right) \kappa
_{\varphi }}{\sqrt{h^{2}\tau _{\psi }^{2}+\kappa _{\varphi }^{2}(1-h)^{2}}}%
=h_{\psi },\ h_{\psi }=\text{constant}
\end{eqnarray*}%
Differentiating (\ref{3.34}), one can easily get
\begin{equation}\label{3.35}
\frac{N}{\kappa }=h_{\varphi }\frac{ds_{\varphi }}{ds}%
\frac{N_{\varphi }}{\kappa _{\varphi }}+h_{\psi }\frac{ds_{\psi }}{ds}\frac{N_{\psi }}{\kappa _{\psi }}. 
\end{equation}%
Hence 
\begin{equation}\label{3.36}
\frac{N}{\kappa }=\frac{N_{\varphi }}{\kappa _{\varphi }}=\frac{N_{\psi }}{%
\kappa _{\psi }}  
\end{equation}%
and 
\begin{equation}\label{3.37}
\kappa =\kappa _{\varphi }h_{\varphi }\frac{ds_{\varphi }}{ds}+\kappa _{\psi
}h_{\psi }\frac{ds_{\psi }}{ds}. 
\end{equation}%
Using (\ref{3.34}) and (\ref{3.36}), we can find
\begin{equation}\label{3.38}
\frac{B}{\kappa }=h_{\varphi }\frac{B_{\varphi }}{\kappa _{\varphi }}+h_{\psi }%
\frac{B_{\psi }}{\kappa _{\psi }}.
\end{equation}%
Differentiating (\ref{3.38}), we have 
\begin{equation}\label{3.39}
\tau \frac{N}{\kappa }=h_{\varphi }\tau _{\varphi }\frac{ds_{\varphi }}{ds}%
\frac{N_{\varphi }}{\kappa _{\varphi }}+h_{\psi }\tau _{\psi }\frac{%
ds_{\psi }}{ds}\frac{N_{\psi }}{\kappa _{\psi }}. 
\end{equation}%
Hence%
\begin{equation}\label{3.40}
\tau =\tau _{\varphi }h_{\varphi }\frac{ds_{\varphi }}{ds}+\tau _{\psi
}h_{\psi }\frac{ds_{\psi }}{ds}.  
\end{equation}
Using (\ref{3.25}) and (\ref{3.26}), we have 
\begin{equation*}
-\frac{\kappa _{\varphi }}{\tau _{\psi }}=\frac{ds_{\psi }}{ds_{\varphi }}=%
\frac{\tau _{\varphi }}{\kappa _{\psi }}\text{ }
\end{equation*}%
and 
\begin{equation}\label{3.41}
\frac{\tau _{\varphi }}{\kappa _{\varphi }}=-\frac{\kappa _{\psi }}{\tau
_{\psi }}.  
\end{equation}
Assume that 
\begin{equation*}
M_{\varphi }=h_{\varphi }\frac{ds_{\varphi }}{ds},\ M_{\psi }=h_{\psi }\frac{%
ds_{\psi }}{ds},
\end{equation*}
then by (\ref{3.37}),(\ref{3.40}) and (\ref{3.41}), we have 
\begin{eqnarray*}
\frac{\kappa }{\tau _{\psi }M_{\psi }}+\frac{\tau }{\kappa _{\varphi
}M_{\varphi }} &=&\frac{\kappa _{\varphi }M_{\varphi }+\kappa _{\psi }M_{\psi
}}{\tau _{\psi }M_{\psi }}+\frac{\tau _{\varphi }M_{\varphi }+\tau _{\psi
}M_{\psi }}{\kappa _{\varphi }M_{\varphi }}\\
&=&\frac{\kappa _{\varphi }}{\tau
_{\psi }}\frac{M_{\varphi }}{M_{\psi }}+\underset{0}{\underbrace{\left( 
\frac{\kappa _{\psi }}{\tau _{\psi }}+\frac{\tau _{\varphi }}{\kappa
_{\varphi }}\right) }}+\frac{\tau _{\psi }}{\kappa _{\varphi }}\frac{M_{\psi
}}{M_{\varphi }} \\
&=&\frac{\kappa _{\varphi }}{\tau _{\psi }}\frac{M_{\varphi }}{M_{\psi }}+%
\frac{\tau _{\psi }}{\kappa _{\varphi }}\frac{M_{\psi }}{M_{\varphi }}=\text{%
constant.}
\end{eqnarray*}%
Since $\frac{\kappa _{\varphi }}{\tau _{\psi }}$ and $\  \frac{M_{\varphi }}{%
M_{\psi }}=\frac{h}{1-h}\frac{ds_{\psi }}{ds_{\varphi }}=\frac{h}{1-h}\frac{%
\kappa _{\varphi }}{\tau _{\psi }}$ are constant and this is the intrinsic
equation of a Bertrand curve.
\end{proof}
\section{Acknowledgment} The authors would like to thank all the anonymous referees for their valuable comments and suggestions which helped to improve this paper.

\end{document}